\documentclass[12pt,reqno]{amsart}

\usepackage{latexsym}
\usepackage{amsmath}
\usepackage{amsfonts}
\usepackage{amssymb}
\usepackage[dvips]{graphicx}
\usepackage{color}
\usepackage{geometry}
\newtheorem{theorem}{Theorem}[section]
\newtheorem*{theorem*}{Theorem}

\theoremstyle{definition}

\theoremstyle{remark}
\newtheorem{remark}[theorem]{Remark}
\numberwithin{equation}{section}

\DeclareMathOperator{\sech}{sech}

\usepackage{lipsum}
\makeatletter
\g@addto@macro{\endabstract}{\@setabstract}
\newcommand{\authorfootnotes}{\renewcommand\thefootnote{\@fnsymbol\c@footnote}}%
\makeatother

\newcommand\blfootnote[1]{%
  \begingroup
  \renewcommand\thefootnote{}\footnote{#1}%
  \addtocounter{footnote}{-1}%
  \endgroup
}

\author{J.~Aguilar}
\address{Department of Mathematics, University of Georgia, Athens, GA 30602.}
\email{jacobaguilar@uga.edu} 
\author{M. ~M.Tom}
\address{Department of Mathematics, Louisiana State University, Baton Rouge, LA 70803.}
\email{tom@math.lsu.edu}

\begin{document}

\begin{center}
  \LARGE 
 Convergence of the solutions of the BBM-KP and the BBM model equations \par \bigskip

  \normalsize
  \authorfootnotes
  Jacob B. Aguilar and Michael ~M. Tom \par\bigskip

\end{center}


\begin{abstract}
It is shown that the solution of the Cauchy problem for the BBM-KP equation converges
to the solution of the Cauchy problem for the BBM equation in a suitable function space whenever
the initial data for both equations are close as the transverse variable $y \rightarrow \pm \infty$.
\end{abstract}

\section{Introduction}

Considered in this paper are the pure-initial-value problem for the Benjamin-Bona-Mahony (BBM) equation
\bigskip
\begin{equation}
\begin{cases}
u_{t} + u_{x} + uu_{x} - u_{xxt} = 0, & x\in \mathbb{R}, t>0,\\
u(x,0)=\phi(x) &
\end{cases}
\label{1.1}
\end{equation}
\bigskip

\noindent and the Benjamin-Bona-Mahony-Kadomtsev-Petviashvili (BBM-KP) equation

\bigskip 
\begin{equation}
\begin{cases}
(\eta_{t} + \eta_{x} + \eta \eta_{x} - \eta_{xxt})_x + \gamma \eta_{yy} = 0, & (x,y)\in\mathbb{R}^2, t>0,\\
\eta(x,y,0)=\psi(x,y)
\end{cases}
\label{1.2}
\end{equation}
\bigskip

where $\gamma = \pm1.$

\bigskip

The BBM (\ref{1.1}), introduced in \cite{BBM}, has proven to be a good approximation for the unidirectional propagation of small amplitude long waves
in a channel where variation across the channel can be safely ignored. 
\\

The BBM-KP (\ref{1.2}) is the regularized version of the usual KP equation which arises in various contexts where nonlinear dispersive waves propagate 
principally along the $x$-axis, but with weak dispersive effects being felt in the direction parallel to the $y$-axis, perpendicular to the main
direction of propagation. 
\\

Our goal in this paper is the show that if the difference between the initial data $\psi(x,y)$ and $\phi(x)$ is small as $y \rightarrow +\infty$, in a certain function space, then the solution to the Cauchy problem (\ref{1.2}) converges to the solution of the Cauchy problem (\ref{1.1})
in a suitable function space, as $y \rightarrow + \infty$. In a similar fashion, the result holds true for $y \rightarrow -\infty.$
\\

The Cauchy problems for both equations have been studied by various authors. Of recent, Bona and Tzvetkov \cite{BT} proved that the Cauchy problem (\ref{1.1}) is 
globally well-posed in $H^s(\mathbb{R})$ for $s \geq 0.$ Bona et al. \cite{BLT} has shown that the Cauchy problem (\ref{1.2}) can be solved by Picard Iteration yielding to local and 
global well-posedness results. In particular, it is shown that the pure initial-value problem (\ref{1.2}), regardless of the sign of $\gamma$, is globally well-posed in

$$W_1(\mathbb{R}^2)=\{\psi \in L^2(\mathbb{R}^2):  ||\psi||_{L^2} + ||\psi_x||_{L^2}  + ||\psi_{xx}||_{L^2}+ ||\partial^{-1}_x \partial_y \psi||_{L^2}+ ||\psi_y||_{L^2} < \infty \}.$$

Saut and Tzvetkov \cite{ST} improved this global well-posedness in the space

$$Y=\{\psi \in L^2(\mathbb{R}^2):\psi_x \in L^2(\mathbb{R}^2)\}.$$

We will employ the following notation. Let $H^s(\mathbb{R}^2)$ denote the classical Sobolev space $H^s(\mathbb{R}^2)$ equipped with the norm

$$||\eta||_s = \Big(\int_{\mathbb{R}^2}(1 + \mu^2 + \xi^2)^s|\hat{\eta}(\xi,\mu)|^2 d\xi d \mu \Big)^\frac{1}{2}.$$

Analogously, $H^k_x(\mathbb{R})$ will denote the Sobolev space in just the spatial variable $x$ with the norm

$$||f||_{H^k_x(\mathbb{R})} = \Big(\int_{\mathbb{R}}(1 + \xi^2)^k|\hat{f}(\xi)|^2 d \xi \Big)^\frac{1}{2}.$$

Define the space

$$H^s_{-1}(\mathbb{R}^2)=\{\eta \in S'(\mathbb{R}^2):||\eta||_{H^s_{-1}(\mathbb{R}^2)} < \infty\}$$

equipped with the norm

$$||\eta||_{H^s_{-1}(\mathbb{R}^2)} = \Big(\int_{\mathbb{R}^2}(1+|\xi|^{-1})^2(1 + \xi^2 + \mu^2)^s|\hat{\eta}(\xi,\mu)|^2 d\xi d \mu \Big)^\frac{1}{2}.$$

We now summarize the existence theory for both initial-value problems. For the proofs refer to  \cite{BT} and \cite{BLT} respectively.

\begin{theorem} 
Fix $s \geq 0$. For any $\phi \in H^s(\mathbb{R})$, there exists a $T=T(||\phi||_{H^s}) > 0$ and a unique solution
$u \in C([0,T]; H^s(\mathbb{R}))$ of the IVP (\ref{1.1}). 

Moreover, for $R>0$, let $B_R$ connote a ball of radius $R$ centered at the origin in $H^s(\mathbb{R})$ and let $T=T(R) >0$ denote a uniform
existence time for the IVP (\ref{1.1}) with $\phi \in B_R$. Then the correspondence $\phi \mapsto u$ which associates to $\phi$ the solution
$u$ of the IVP (\ref{1.1}) with initial data $\phi$ is a real analytic mapping of $B_R$ to $C([T,-T]; H^s(\mathbb{R}))$.
\end{theorem}

The above theorem improves the earlier known result proven by Benjamin et. al. \cite{BBM} where the IVP (\ref{1.1}) was shown to be
globally well-posed for data in $H^k$, $k\in \mathbb{Z}$ such that $k \geq 1$. For the Cauchy problem (\ref{1.2}), we have

\begin{theorem}
 Let $\psi \in H^s_{-1}(\mathbb{R}^2)$ with $s > \frac{3}{2}$. Then there exist a $T_0$ such that the initial-value problem 
(\ref{1.2}) has a unique solution $\eta \in C([0,T]; H^s_{-1}(\mathbb{R}^2))$, $\partial^{-1}_x\eta_y \in C([0,T]; H^{s-1}_{-1}(\mathbb{R}^2))$,
with \ \  $\eta_t \in C([0,T]; H^{s-2}(\mathbb{R}^2))$. Moreover, the map $\psi \rightarrow \eta$ is continuous from  $H^s_{-1}(\mathbb{R}^2)$
to $C([0,T_0]; H^s_{-1}(\mathbb{R}^2)).$

\end{theorem}

\section{Main Result}

Assume that $\psi \in H^s_{-1}(\mathbb{R}^2)$ with $s > \frac{3}{2}$, and let $u^+$ be the solution to the initial-value problem (\ref{1.1}) corresponding to the initial data $\phi^+$ where

$$\phi^+(\cdot)= \lim_{y \to +\infty} \psi(\cdot,y),$$

\noindent and $u^-$ be the solution corresponding to the initial data $\phi^-$ where

$$\phi^-(\cdot)= \lim_{y \to -\infty} \psi(\cdot,y).$$ 

For example, we could let $\psi(x,y) = \phi(x) + \sech y$, where $\phi \in H^s(\mathbb{R}).$

Let $\eta$ be the solution to the initial-value problem (\ref{1.2}) corresponding to the initial data $\psi.$

Define the function

\begin{equation} 
w(x,y,t) = \eta - \frac{1}{2} [u^{+} + u^{-}] - \frac{1}{2}[u^{+} - u^{-}] \tanh y.
\end{equation}

Observe that 

\begin{equation}
 w(x,y,0) = \psi(x,y) - \frac{1}{2} [\phi^{+}(x) + \phi^{-}(x)] - \frac{1}{2}[\phi^{+}(x) - \phi^{-}(x)] \tanh y
\end{equation}

and hence

$$\lim_{y \to \pm \infty} w(x,y,0)=0.$$

A straightforward calculation shows that $w$ satisfies the following initial value problem

{\small{
\begin{equation}
\begin{cases}
w_t + w_x -w_{xxt} - \partial^{-1}_x \eta_{yy} + w w_x + \frac{1}{2}(1 + \tanh y) (u^+ w)_x + \frac{1}{2}(1- \tanh y)(u^- w)_x& \\
 - \frac{1}{4} ( 1 - \tanh^2 y) ( u^+ u^+_x + u^- u^-_x - u^+u^-_x - u^-u^+_x) = 0, & \\
\\
w(x,y,0)= \psi(x,y) - \frac{1}{2}[ \phi^+ +\phi^-] - \frac{1}{2}[\phi^+ -\phi^-] \tanh y.&
\end{cases}
\label{mon}
\end{equation}}}

\bigskip

\begin{remark}

It should be noted that, $\partial^{-1}_x f$ is defined via the Fourier transform as 

$$\widehat{\partial^{-1}_x f} = \frac{1}{i\xi} \hat{f}(\xi,y).$$

Due to the singularity of the symbol $\frac{1}{\xi}$ at $\xi =0$, one requires that $\hat{f}(0,y)=0$
(the Fourier transform in the variable $x$), which is clearly equivalent to

$$\int_{\mathbb{R}}f(x,y) \,dx=0.$$

In what follows, $\partial^{-1}_x f\in L^2(\mathbb{R}^2)$ means there is an $L^2(\mathbb{R}^2)$ function $g$ such 
that $g_x=f$, at least in the distributional sense. We now state our main theorem.

\end{remark}

\begin{theorem}
Let $\phi^{\pm} \in H^k(\mathbb{R})$ be such that $\phi^\pm(\cdot)= \lim_{y \to \pm\infty} \psi(\cdot,y)$ where $\psi \in H^s_{-1}(\mathbb{R}^2)$ 
with $k\geq 1$ and $s\geq k+1.$ Also, let $u^\pm$ and $\eta$ both be solutions of (\ref{1.1}) and (\ref{1.2}) respectively which are emanating from Theorem 1.1 and 1.2 respectively. If

$$\lim_{y \rightarrow \pm \infty}||\psi(\cdot,y)-\phi^\pm(\cdot)||_{H^k_x(\mathbb{R})}=0,$$
then
$$\lim_{y \rightarrow \pm \infty}||\eta -u^\pm||_{H^k_x(\mathbb{R})}=0.$$
\end{theorem}

\begin{proof}
We first estimate $||w||_{H^1_x(\mathbb{R})}$ for any $y \in \mathbb{R}$. Multiply equation (\ref{mon}) by $w$ and integrate over $\mathbb{R}$ in the spatial variable $x$
to obtain the following integral

$$\int_{\mathbb{R}} [w w_t + w w_x - w w_{xxt} -  w \partial^{-1}_x \eta_{yy} + w^2 w_x + w \frac{1}{2}(1 + \tanh y)  (u^+ w)_x +w \frac{1}{2}(1- \tanh y)(u^- w)_x $$

$$-w \frac{1}{4} ( 1 - \tanh^2 y) ( u^+ u^+_x + u^- u^-_x - u^+u^-_x - u^-u^+_x)]\,dx =0.$$

\bigskip
\bigskip
\bigskip
\bigskip

After a few integration by parts, we arrive at the following estimate

$$\frac{1}{2} \frac{d}{dt}  \Big[\int_{\mathbb{R}}w^2 dx + \int_{\mathbb{R}}w^2_x dx\Big] \leq \Big|\int_{\mathbb{R}} w \partial^{-1}_x \eta_{yy} dx \Big | + \frac{1}{2}\Big | \int_{\mathbb{R}} (1 + \tanh y) u^+w_xw dx \Big | $$

$$ + \frac{1}{2}\Big |\int_{\mathbb{R}}(1 - \tanh y) u^- w_xw  dx \Big |+ \frac{1}{4}\Big| \int_{\mathbb{R}} ( 1 - \tanh^2 y)w  u^+ u^+_x\,dx\Big|+ \frac{1}{4}\Big| \int_{\mathbb{R}} ( 1 - \tanh^2 y)w  u^- u^-_x dx \Big|$$

$$ +\frac{1}{4} \Big | \int_{\mathbb{R}}  ( 1 - \tanh^2 y) w  u^+ u^-_x dx \Big | + \frac{1}{4} \Big |\int_{\mathbb{R}} ( 1 - \tanh^2 y)w  u^- u^+_x dx \Big |.$$

\bigskip 
 
Making use of H$\ddot{o}$lders inequality, it follows that 
\bigskip

$$\frac{1}{2} \frac{d}{dt} || w ||^2_{H^1_x(\mathbb{R})} \leq ||w||_{L^2(\mathbb{R})}  ||\partial^{-1}_x \eta_{yy}||_{L^2(\mathbb{R})} + \frac{(1 + \tanh y)}{2}| u^+|_\infty||w||_{L^2(\mathbb{R})}||w_x||_{L^2(\mathbb{R})}$$

$$ + \frac{(1 - \tanh y)}{2} | u^-|_\infty||w||_{L^2(\mathbb{R})}||w_x||_{L^2(\mathbb{R})}+ \frac{( 1 - \tanh^2 y)}{4}  \text{ } | u^+|_\infty||w||_{L^2(\mathbb{R})}||u^+_x||_{L^2(\mathbb{R})}$$

$$ + \frac{( 1 - \tanh^2 y)}{4}  \text{ }  | u^-|_\infty||w||_{L^2(\mathbb{R})}||u^-_x||_{L^2(\mathbb{R})}+ \frac{( 1 - \tanh^2 y)}{4}  \text{ }  | u^+|_\infty||w||_{L^2(\mathbb{R})}||u^-_x||_{L^2(\mathbb{R})}$$

$$ + \frac{( 1 - \tanh^2 y)}{4}  \text{ }  | u^-|_\infty||w||_{L^2(\mathbb{R})}||u^+_x||_{L^2(\mathbb{R})}.$$

\bigskip
\bigskip
An application of Young's inequality yields
\bigskip
\bigskip

$$\frac{1}{2} \frac{d}{dt} || w ||^2_{H^1_x(\mathbb{R})} \leq ||w||_{H^1_x(\mathbb{R})}  ||\partial^{-1}_x \eta_{yy}||_{L^2(\mathbb{R})} +  \frac{(1 + \tanh y)}{2}| u^+|_\infty||w||^2_{H^1_x(\mathbb{R})} $$

$$ +  \frac{(1 - \tanh y)}{2}| u^-|_\infty||w||^2_{H^1_x(\mathbb{R})} + \frac{( 1 - \tanh^2 y)}{4}  \text{ } | u^+|_\infty||w||_{H^1_x(\mathbb{R})}||u^+_x||_{L^2(\mathbb{R})}$$

$$ + \frac{( 1 - \tanh^2 y)}{4} \text{ }  | u^-|_\infty||w||_{H^1_x(\mathbb{R})}||u^-_x||_{L^2(\mathbb{R})}+ \frac{( 1 - \tanh^2 y)}{4} \text{ }  | u^+|_\infty||w||_{H^1_x(\mathbb{R})}||u^-_x||_{L^2(\mathbb{R})}$$

$$ + \frac{( 1 - \tanh^2 y)}{4} \text{ }  | u^-|_\infty||w||_{H^1_x(\mathbb{R})}||u^+_x||_{L^2(\mathbb{R})}.$$

\bigskip
\bigskip
After utilizing an elementary embedding theorem, the last inequality can be written in the form
\bigskip
\bigskip

$$\frac{1}{2} \frac{d}{dt} || w ||^2_{H^1_x(\mathbb{R})} \leq \Big[ ||\partial^{-1}_x \eta_{yy}||_{L^2(\mathbb{R})}+ ( 1 - \tanh^2 y) a\Big(|| u^+ ||_{H^1(\mathbb{R})}, || u^-||_{H^1(\mathbb{R})}  \Big) \Big]|| w ||_{H^1_x(\mathbb{R})}$$

$$+\frac{1}{2}\Big[ (1 + \tanh y)||u^+||_{H^1(\mathbb{R})}  + (1 - \tanh y)||u^-||_{H^1(\mathbb{R})}  \Big]|| w ||^2_{H^1_x(\mathbb{R})},$$

\bigskip
\bigskip
or
\bigskip
\bigskip

$$ \frac{1}{2} \frac{d}{dt} || w ||^2_{H^1_x(\mathbb{R})} \leq (C_\eta + C_1)|| w ||_{H^1_x(\mathbb{R})} + (C_+ + C_-)||w||^2_{H^1_x(\mathbb{R})}.$$

\bigskip
\bigskip
From this the following inequality is derived
\bigskip
\bigskip

$$ \frac{d}{dt} || w ||_{H^1_x(\mathbb{R})} \leq (C_\eta + C_1) + (C_+ + C_-)||w||_{H^1_x(\mathbb{R})}.$$

\bigskip
\bigskip

By a variant of Gronwall's lemma, it follows that
\bigskip
\bigskip

 \begin{equation}
 || w ||_{H^1_x(\mathbb{R})}\leq || w(x,y,0) ||_{H^1_x(\mathbb{R})} e^{(C_+ + C_-)t} +  \frac{C_\eta +C_1}{C_+ + C_-}\Big(e^{(C_+ + C_-)t} -1\Big).
 \end{equation}

\bigskip
\bigskip

By letting $y \rightarrow + \infty$, we observe that $C_\eta =||\partial^{-1}_x \eta_{yy}||_{L^2(\mathbb{R})} \rightarrow 0$
since, for $s\geq 2$, $\partial^{-1}_x \eta_{yy} \rightarrow 0$ as $y \rightarrow \pm \infty.$ In addition, it follows that

$$C_1 =(1 - \tanh^2 y)a\Big(||u^+||_{H^1(\mathbb{R})}, ||u^-||_{H^1(\mathbb{R})} \Big) \rightarrow 0,$$

\medskip

since $(1 - \tanh^2 y) \rightarrow 0$ as $y \rightarrow \pm \infty$ and $a\Big(||u^+||_{H^1(\mathbb{R})}, ||u^-||_{H^1(\mathbb{R})} \Big) < \infty$ by the global well-posedness of the BBM.
Clearly $C_- =\frac{1}{2}(1 -\tanh y)||u^-||_{H^1(\mathbb{R})}\rightarrow 0$ as $y \rightarrow +\infty$, but $C_+$ does not.

\bigskip

Let $y \rightarrow + \infty$ on both sides of inequality (2.4), to conclude

\bigskip

$$\lim _{y\rightarrow + \infty} || w ||_{H^1_x(\mathbb{R})}\leq 0 \cdot e^{C_+t} + 0\cdot\Big(e^{C_+t} -1\Big)=0.$$

\bigskip

We have established that

\bigskip

$$\lim _{y\rightarrow + \infty} || w ||_{H^1_x(\mathbb{R})} =\lim_{y \rightarrow + \infty} || \eta(x,y,t) - u^{+} (x,t)||_{H^1_x(\mathbb{R})}=0.$$

\bigskip
\bigskip

A similar reasoning follows in the case when $y \rightarrow - \infty$, since $C_\eta \rightarrow 0$ 
and $C_1 \rightarrow 0.$ However, in this case $C_+ =\frac{1}{2}(1 +\tanh y)||u^+||_{H^1(\mathbb{R})} \rightarrow 0$, but $C_-$ does not. Hence,

$$\lim_{y \rightarrow - \infty} || \eta(x,y,t) - u^{-} (x,t)||_{H^1_x(\mathbb{R})}=\lim _{y\rightarrow - \infty} || w ||_{H^1_x(\mathbb{R})}\leq 0 \cdot e^{C_-t} + 0\cdot\Big(e^{C_-t} -1\Big)=0.$$

\bigskip

Therefore, 

\bigskip

$$\lim_{y \rightarrow \pm \infty} || \eta(x,y,t) - u^{\pm} (x,t)||_{H^1_x(\mathbb{R})}=0.$$

\bigskip
\bigskip
\bigskip
\bigskip
\bigskip

More generally, for $k \geq 1$, we apply the operator $\partial^k_x$ to both sides of the differential equation (\ref{mon}), multiply the result by $\partial^k_x w$,
and integrate the result over $\mathbb{R}$ in the spatial variable $x$. After a few integration by parts we arrive at the following integral equation

\bigskip
\bigskip

$$\frac{1}{2}\frac{d}{dt} \Big[ \int_{\mathbb{R}}   (\partial^k_xw)^2 \,dx + \int_{\mathbb{R}}  (\partial^{k+1}_xw)^2 \,dx \Big] = \int_{\mathbb{R}} \partial^k_xw \partial^{k-1}_x \eta_{yy}\,dx$$

$$+\frac{(-1)^{k}}{2}(1+ \tanh y) \int_{\mathbb{R}} \Big(\partial^{2k+1}_xw \Big)(u^+w) \,dx + \frac{(-1)^{k}}{2}(1-\tanh y) \int_{\mathbb{R}} \Big(\partial^{2k+1}_xw \Big)(u^-w) \,dx$$

$$+ \frac{(-1)^k}{4}( 1 - \tanh^2 y)\int_{\mathbb{R}} \Big(\partial^{2k}_x w \Big)(u^+ u^+_x)\,dx+ \frac{(-1)^k}{4}( 1 - \tanh^2 y)\int_{\mathbb{R}} \Big(\partial^{2k}_x w \Big)(u^- u^-_x)\,dx$$

$$ +\frac{(-1)^{k+1}}{4}( 1 - \tanh^2 y)\int_{\mathbb{R}} \Big(\partial^{2k}_x w \Big)(u^+ u^-_x)\,dx+\frac{(-1)^{k+1}}{4}( 1 - \tanh^2 y)\int_{\mathbb{R}} \Big(\partial^{2k}_x w \Big)(u^- u^+_x)\,dx.$$  

\bigskip
\bigskip

Similar to the case for $H^1$, the above equation delivers the bound

$$\frac{1}{2} \frac{d}{dt} \Big[|| \partial^k_xw||^2_{L^2(\mathbb{R})} + || \partial^{k+1}_xw||^2_{L^2(\mathbb{R})} \Big ] \leq\int_{\mathbb{R}} |\partial^k_xw \partial^{k-1}_x \eta_{yy} | \,dx  $$

$$+  \frac{( 1 + \tanh y)}{2} |u^+|_\infty  \int_{\mathbb{R}} \Big |\partial^k_x w \partial^{k+1}_xw \Big|\,dx +  \frac{( 1 - \tanh y)}{2}| u^-|_\infty\int_{\mathbb{R}} \Big |\partial^k_x w \partial^{k+1}_xw \Big |\,dx$$

$$ +  \frac{( 1 - \tanh^2 y)}{4} | u^+|_\infty\int_{\mathbb{R}} \Big |\partial^{k+1}_xw    \partial^k_x u^+ \Big |\,dx  +  \frac{( 1 - \tanh^2 y)}{4} | u^-|_\infty \int_{\mathbb{R}} \Big|\partial^{k+1}_xw    \partial^k_x u^- \Big |\,dx$$
 
$$+  \frac{( 1 - \tanh^2 y)}{4}| u^+|_\infty\int_{\mathbb{R}}\Big | \partial^{k+1}_xw   \partial^k_x u^-\Big | \,dx  +   \frac{( 1 - \tanh^2 y)}{4}| u^-|_\infty \int_{\mathbb{R}} \Big|\partial^{k+1}_xw    \partial^k_x u^+\Big| \,dx .$$

\bigskip
An appeal to  H$\ddot{o}$lders inequality results in

\bigskip
 
$$\frac{1}{2} \frac{d}{dt} \Big[|| \partial^k_xw||^2_{L^2(\mathbb{R})} + || \partial^{k+1}_xw||^2_{L^2(\mathbb{R})} \Big ] \leq ||\partial^k_xw||_{L^2(\mathbb{R})}  ||\partial^{k-1}_x \eta_{yy}||_{L^2(\mathbb{R})}$$

$$ + \frac{( 1 + \tanh y)}{2} | u^+|_\infty||\partial^k_xw||_{L^2(\mathbb{R})}||\partial^{k+1}_xw||_{L^2(\mathbb{R})} +  \frac{( 1 - \tanh y)}{2}| u^-|_\infty||\partial^k_xw||_{L^2(\mathbb{R})}||\partial^{k+1}_xw||_{L^2(\mathbb{R})}$$

$$+ \frac{( 1 - \tanh^2 y)}{4} \text{ } | u^+|_\infty||\partial^{k+1}_xw||_{L^2(\mathbb{R})}||\partial^k_xu^+||_{L^2(\mathbb{R})}$$

$$+ \frac{( 1 - \tanh^2 y)}{4}\text{ }  | u^-|_\infty||\partial^{k+1}_xw||_{L^2(\mathbb{R})}||\partial^k_xu^-||_{L^2(\mathbb{R})}$$

$$+ \frac{( 1 - \tanh^2 y)}{4}\text{ }  | u^+|_\infty||\partial^{k+1}_xw||_{L^2(\mathbb{R})}||\partial^k_xu^-||_{L^2(\mathbb{R})}$$

$$ + \frac{( 1 - \tanh^2 y)}{4}\text{ }  | u^-|_\infty||\partial^{k+1}_xw||_{L^2(\mathbb{R})}||\partial^k_xu^+||_{L^2(\mathbb{R})}.$$

\bigskip
After invoking Young's inequality, we have
\bigskip

$$ \frac{1}{2} \frac{d}{dt} \Big[|| \partial^k_xw||^2_{L^2(\mathbb{R})} + || \partial^{k+1}_xw||^2_{L^2(\mathbb{R})} \Big ] \leq \Big[|| \partial^k_xw||^2_{L^2(\mathbb{R})} + || \partial^{k+1}_xw||^2_{L^2(\mathbb{R})} \Big ]^\frac{1}{2}  ||\partial^{k-1}_x \eta_{yy}||_{L^2(\mathbb{R})}$$

$$ +  \frac{( 1 + \tanh y)}{2}| u^+|_\infty\Big[|| \partial^k_xw||^2_{L^2(\mathbb{R})} + || \partial^{k+1}_xw||^2_{L^2(\mathbb{R})} \Big ]$$

$$  + \frac{( 1 - \tanh y)}{2}| u^-|_\infty\Big[|| \partial^k_xw||^2_{L^2(\mathbb{R})} + || \partial^{k+1}_xw||^2_{L^2(\mathbb{R})} \Big ]$$

$$ + \frac{( 1 - \tanh^2 y)}{4} \text{ } | u^+|_\infty\Big[|| \partial^k_xw||^2_{L^2(\mathbb{R})} + || \partial^{k+1}_xw||^2_{L^2(\mathbb{R})} \Big ]^\frac{1}{2}||\partial^k_xu^+||_{L^2(\mathbb{R})}$$

$$ + \frac{( 1 - \tanh^2 y)}{4}\text{ }  | u^-|_\infty\Big[|| \partial^k_xw||^2_{L^2(\mathbb{R})} + || \partial^{k+1}_xw||^2_{L^2(\mathbb{R})} \Big ]^\frac{1}{2}||\partial^k_xu^-||_{L^2(\mathbb{R})}$$

$$+ \frac{( 1 - \tanh^2 y)}{4}\text{ }  | u^+|_\infty\Big[|| \partial^k_xw||^2_{L^2(\mathbb{R})} + || \partial^{k+1}_xw||^2_{L^2(\mathbb{R})} \Big ]^\frac{1}{2}||\partial^k_xu^-||_{L^2(\mathbb{R})}$$

$$ + \frac{( 1 - \tanh^2 y)}{4}\text{ }  | u^-|_\infty\Big[|| \partial^k_xw||^2_{L^2(\mathbb{R})} + || \partial^{k+1}_xw||^2_{L^2(\mathbb{R})} \Big ]^\frac{1}{2}||\partial^k_xu^+||_{L^2(\mathbb{R})}.$$

\bigskip
\bigskip
\bigskip

From which we deduce that

\bigskip

$$\frac{1}{2} \frac{d}{dt} || w ||_{H^k_x(\mathbb{R})}^2 \leq \Big[ ||\partial^{k-1}_x \eta_{yy}||_{L^2(\mathbb{R})}+( 1 - \tanh^2 y) a\Big(||\partial^k_xu^+||_{H^1(\mathbb{R})}, ||\partial^k_xu^-||_{H^1(\mathbb{R})} \Big)\Big]|| w ||_{H^k_x(\mathbb{R})}$$

$$+\frac{1}{2}\Big[ (1 + \tanh y)||u^+||_{H^1(\mathbb{R})}+ ( 1 - \tanh y)||u^-||_{H^1(\mathbb{R})}\Big]|| w ||_{H^k_x(\mathbb{R})}^2.$$

\bigskip
\bigskip
This leads us to the following inequality
\bigskip
\bigskip
$$ \frac{d}{dt} || w ||_{H^k_x(\mathbb{R})} \leq (C_\eta + C_1) + (C_+ + C_-)||w||_{H^k_x(\mathbb{R})}.$$

\bigskip
\bigskip
Proceeding with the variant of Gronwall's lemma, it follows that

\bigskip
\bigskip

 \begin{equation}
 || w ||_{H^k_x(\mathbb{R})}\leq || w(x,y,0) ||_{H^k_x(\mathbb{R})} e^{(C_+ + C_-)t} +  \frac{C_\eta +C_1}{C_+ + C_-}\Big(e^{(C_+ + C_-)t} -1\Big).
 \end{equation}

Similarly as in the case of $H^1$, we let $y \rightarrow \pm \infty$ on both sides of inequality (2.5), to conclude

\bigskip

$$\lim _{y\rightarrow \pm \infty} || w ||_{H^k_x(\mathbb{R})}\leq 0 \cdot e^{C_{\pm}t} + 0\cdot\Big(e^{C_{\pm}t} -1\Big)=0.$$

\bigskip
\bigskip

Where all of the constants vanish except for $C_+$ as $y\rightarrow + \infty$ also, as $y\rightarrow - \infty$, $C_-$ does not vanish.
Placing this together, we conclude that

\bigskip

$$\lim _{y\rightarrow \pm \infty} || w ||_{H^k_x(\mathbb{R})} =\lim_{y \rightarrow \pm \infty} || \eta(x,y,t) - u^{\pm} (x,t)||_{H^k_x(\mathbb{R})}=0.$$

\end{proof}


\begin{thebibliography}{99}
\bibitem{BBM} T.B. Benjamin, J.L. Bona, and J.J. Mahony, {\it Model equations for long waves in nonlinear dispersive
systems}, Philos. Trans. Roy. Soc. London Ser. A 272 (1972), 47-78. 

\bibitem{BLT} J. L. Bona, Y. Liu, and M. M. Tom, {\it The Cauchy problem and stability of solitary wave solutions for RLW-KP 
type equations}, J. Differential Equations 185, no. 2, (2002), 437-482.
%
\bibitem{BPS} J.L. Bona, W. G. Pritchard and L. R. Scott, {\it An evaluation of a model equation for water waves}, 
Philos. Trans. Roy. Soc. London Ser. A 302 (1981), 457-510.

\bibitem{BT} J. L. Bona, N. Tzvetkov, {\it Sharp well-posedness results for the BBM equation}, 
Discrete and Continuous Dynamical Systems
 23 4 (2009), 1241-1252. 

%
\bibitem{HS} J. Hammack and H. Segur, {\it
The Kortweg-de Vries equation and water waves, II. Comparison with experiments}, 
J. Fluid Mech 65 (1974), 289-313.

\bibitem{ST} J.-C. Saut, N. Tzvetkov, Global well-posedness for the KP-BBM equations, AMRX Appl. Math. Res. Express 1
(2004), 1-16.

%
\bibitem{ZG} N. J. Zabusky and C. Galvin, {\it Shallow-water waves, the Kortweg-de Vries equation solitons},
J. Fluid Mech. 47 (1971), 811-824.
%
%

%
%
\end{thebibliography}
\end{document}